\documentclass{amsart}

\usepackage[foot]{amsaddr}
\usepackage[utf8]{inputenc}
\usepackage[T1]{fontenc}
\usepackage{amssymb,latexsym}
\usepackage{amsmath}
\usepackage{amsthm}
\usepackage{amssymb}
\usepackage{graphicx}
\usepackage{color}
\usepackage[shortlabels]{enumitem}
\usepackage[colorlinks=true,citecolor={blue},draft=false, urlcolor={black}]{hyperref}
\usepackage{cleveref}

\usepackage{float}
\usepackage{todonotes}
\usepackage[shortlabels]{enumitem}
\usepackage{comment}

\def\Z{\mathbb Z}

\def\C{\mathbb C}
\def\R{\mathbb R}

\def\N{\mathbb N}
\def\A{\mathcal A}
\def\B{\mathcal B}

\def\Z{\mathbb Z}
\def\R{\mathbb R}

\def\N{\mathbb N}
\def\A{\mathcal A}
\def\B{\mathcal B}

\newtheorem{thm}{Theorem}

\newtheorem{coro}[thm]{Corollary}

\newtheorem{lem}[thm]{Lemma}
\newtheorem{lmx
m}[thm]{Lemma}
\newtheorem{claim}[thm]{Claim}

\newtheorem{prop}[thm]{Proposition}
\newtheorem{defi}[thm]{Definition}

\crefname{thm}{theorem}{theorems}
\crefname{theorem}{theorem}{theorems}
\crefname{coro}{corollary}{corollaries}
\crefname{example}{example}{examples}
\crefname{lem}{lemma}{lemmas}
\crefname{lmm}{lemma}{lemmas}
\crefname{claim}{claim}{claims}
\crefname{obs}{observation}{observations}
\crefname{proposition}{proposition}{propositions}
\crefname{prop}{proposition}{propositions}
\crefname{defi}{definition}{definitions}

\newtheorem{example}[thm]{Example}

\crefname{example}{example}{examples}

\begin{document}

\title{On positional representation of integer vectors}

\author[Edita Pelantov\'a]{Edita \textsc{Pelantová}}
\author[Tomáš Vávra]{Tomáš \textsc{Vávra}}

\address[Edita Pelantov\'a]{
Department of Mathematics, FNSPE Czech Technical University in Prague\\
Trojanova 13, 120 00 Praha 2, Czech Republic}

\address[Tomáš Vávra]{Department of Pure Mathematics, University of Waterloo, Waterloo, Ontario, Canada N2L 3G1}

\begin{abstract}

We show that any $m\times m$  matrix $M$ with integer entries  and $\det M  =\Delta \neq 0$ can be equipped by a finite digit set  $\mathcal{D}\subset\mathbb{Z}^m$
such that any integer $m$-dimensional vector  belongs to the set
$$  {\rm Fin}_{\mathcal{D}}(M)= \Bigl\{\sum_{k\in I}M^k {d}_k : \emptyset\neq I  \text{ finite subset of } \mathbb{Z} \text{ and } {d}_k \in \mathcal{D} \text{ for each }  k \in I\Bigr\} \subset \bigcup\limits_{k\in \mathbb{N}} \frac{1}{\Delta^k}\mathbb{Z}^{m} \,.
$$
We  also characterize  the matrices $M$  for which  the sets  $ {\rm Fin}_{\mathcal{D}}(M)$   and $ \bigcup\limits_{k\in \mathbb{N}} \frac{1}{\Delta^k}\mathbb{Z}^{m}$  coincide.

\end{abstract}

\maketitle

\noindent \textit{Keywords:} { vector representation, number system, Jordan form  }

\noindent \textit{2000MSC:} 11C20, 11A63, 15B36



\section{Introduction}
 
The idea to represent $m$-dimensional vectors by a single string of digits  can be traced back to the work of  A. Vince \cite{Vince93a}  and \cite{Vince93b} who showed  that for any expansive  matrix  $M \in \mathbb{Z}^{m\times m}$   there exists a digit set $\mathcal{D}\subset \mathbb{Z}^m$ such that any integer vector $x \in \mathbb{Z}^m$ can be written in the form   $x= \sum_{k=0}^nM^k {d}_k$, where $d_k\in \mathcal{D}$. In other words, the whole $\mathbb{Z}^m$ is representable in the matrix numeration system $(M,\mathcal{D})$. 
On the other hand, if a matrix $M$ has an eigenvalue inside the unit circle, no choice of the digit set $\mathcal{D}\subset \mathbb{Z}^m$ allows to represent all integer vectors as a combination of non-negative powers of $M$ only.

 Many works devoted to positional representations of elements of a commutative finitely generated ring can be interpreted as a special case of the matrix numeration systems.   From this point of view, the  history of the matrix numeration systems had started several decades before the year 1993, the year Vince's results have been published.
 Positional number systems used to represent  the Gaussian integers or,  more general,   elements of a ring of integers in a quadratic number field can be viewed as a matrix number system given by  a  $2\times 2$ matrix.  Number systems of this type   were studied by Penney \cite{Penney}, K\'atai and Szab\'o \cite{KataiSzabo75},  K\'atai and B. Kov\'acs,  \cite{KataiKovacs80}, \cite{KataiKovacs81},  Gilbert \cite{Gilbert81}. 
This concept was extended to algebraic fields of higher order by  B. Kov\'acs \cite{KovacsBela81} and   B. Kov\'acs and Peth\H{o} \cite{KovacsPetho83}. A number system with an algebraic base $\beta$ can be interpreted as a matrix number system with the base being the companion matrix of the minimal polynomial of $\beta$. Hence the characteristic polynomial of the matrix is irreducible over $\mathbb{Q}$.  
The concept of the so-called canonical number system for irreducible polynomials   was further  generalized to arbitrary  polynomials from  $\mathbb{Z}[x]$
by Peth\H{o} \cite{Petho91}.
The dynamic properties of canonical number systems can be well studied in the formalism of the shift radix systems introduced in \cite{ABBPT05}. 
A very general setting  of the canonical number systems was recently  considered in  \cite{EvertseSpol19}.

Our aim is to study the matrix numeration systems. The most important property of a square   matrix  can be deduced from its Jordan  form.  
Let us point out that the Jordan form of the companion matrix associated to  a polynomial  has a specific property: only one Jordan block corresponds  to each  eigenvalue. Therefore the study of  matrix numeration systems can display  new phenomena and  is of its own interest.  The  matrix formalism for numeration systems -- under the name {\it numeration systems in lattices}  --  was systematically used  by A. Kov\'acs in  \cite{KovacsAttila03}.  Of course,  A. Kov\'acs, just like Vince, considers integer matrices as they  map  a lattice into itself.   
The Vince's results on integer matrices were recently generalized  by J.  Jankauskas and  J. Thuswaldner to matrices   $M \in   \mathbb{Q}^{m\times m}$ with rational entries and without eigenvalues in modulus strictly smaller than 1, see \cite{JT}.

In this  paper,   we generalize the Vince's result in another direction:  our  representation of an integer  vector  can use both positive and negative powers of a matrix $M \in \mathbb{Z}^{m\times m}$ with  $\det M = \Delta \neq 0$. Our aim is to study the set  $$  {\rm Fin}_{\mathcal{D}}(M)= \Bigl\{\sum_{k\in I}M^k {d}_k : I  \text{ finite non-empty subset of } \mathbb{Z} \text{ and } {d}_k \in \mathcal{D} \text{ for each }  k \in I\Bigr\} \subset \bigcup\limits_{k\in \mathbb{N}} \frac{1}{\Delta^k}\mathbb{Z}^{m} \,.
$$
Notice that because the set of positions $I$ is finite, we can treat the position not in $I$ as being occupied by the zero digit. Therefore we always assume that zero is in the digit set $\mathcal D.$
We show that for any non-singular matrix $M \in \mathbb{Z}^{m\times m}$ there exists  a finite set  $\mathcal{D}\subset \mathbb{Z}^m$ such that
$  \mathbb{Z}^{m} \subset   {\rm Fin}_{\mathcal{D}}(M)$.    
 For such $\mathcal D$, we also characterise  the matrices $M$  for which  the sets  $ {\rm Fin}_{\mathcal{D}}(M)$   and $ \bigcup\limits_{k\in \mathbb{N}} \frac{1}{\Delta^k}\mathbb{Z}^{m}$  coincide.

\section{The main result}
This section contains the main result and its proof that relies heavily on Proposition~\ref{skoroHotovo}. The proof of this proposition is, however, quite technical, therefore we provide it in its own Section \ref{sectionProposition}.

The method of proving the main theorem is distinct from the one used in \cite{JT}. Their setting allowed the problem to be reduced to a problem in canonical number systems. This approach does not work in our case because we allow eigenvalues that are smaller than one in absolute value.

\begin{thm}\label{main}  Let  $M \in \mathbb{Z}^{m\times m}$  be a non-singular matrix. Then there exists a finite digit set $\mathcal{A}\subset \Z^m$ such that any vector $z \in \Z^m$ can be written in the form
$$
z = \sum_{k\in I}  M^kd_k, \ \  \text{ where  }\quad \emptyset\neq I \subset \Z, \ I \text{\ finite\ \ and }\  d_k \in \mathcal{A} \ \text{for each } k \in I.
$$
\end{thm}

The property $\Z^m\subseteq \mathrm{Fin}_{\mathcal D}(M)$ easily implies a stronger property. Because $\mathrm{Fin}_{\mathcal D}(M)$ is closed under multiplication by $M^i$ for $i\in\Z,$ we easily obtain the following corollary.

\begin{coro}
Let $\Z^m\subseteq \mathrm{Fin}_{\mathcal D}(M)$, then we have $\bigcup_{n\in\Z}M^n\Z^m=\mathrm{Fin}_{\mathcal D}(M)$.

In particular, for each $M\in\Z^{m\times m}$ non-singular, there exists a finite digit set $\mathcal D\subset \Z^m$, such that $\mathrm{Fin}_{\mathcal D}(M)=\bigcup_{n\in\Z}M^n\Z^m$.
\end{coro}

The full proof of Theorem~\ref{main} is rather technical. In order to simplify readability, we put the technical part into Proposition~\ref{skoroHotovo} whose proof will be the content of Section~\ref{sectionProposition}. Nevertheless, we can now present the main idea of our work.

For $\mathcal D\subset\Z^n$ finite, define $T_d(x) = Mx-d$ and $x_k= T_d(x_{k-1})$ with $x_0 = x.$ We directly obtain
\begin{equation}\label{finite}
    x = M^{-1}d_1 + M^{-2}d_2 + \dots + M^{-k}d_k + M^{-k}x_k.
\end{equation}
Assume that $S\subseteq\Z^m$ consists of vectors with the property that there exist $d_1,d_2,\dots$ such that $\|x_i\|<C$ eventually (with the constant being universal for the whole $S$). Then the digit set $\mathcal A = \mathcal D+\{d\in\Z^n : \|d\|<C\}$ suffices for \eqref{finite} to be a finite representation of members of $S$ over $\A$.
It then suffices to show that for any $x\in\Z^m$, $M^{-n}x\in S$ for some $n\in\N$. Indeed, we would have $M^{-n}x = \sum_{i=1}^k M^{-i}d_i,$ i.e. $x = \sum_{i=1}^k M^{n-i}d_i.$

In the following, we will show that an appropriate choice of $S$ is all the elements of $\Z^m$ whose projection into the expansive eigenspace is bounded by a constant depending on $M$ only. Such a choice ensures that $M^{-n}x\in S$ is achievable for all $x\in\Z^m$. It then remains to be shown that with a suitable digit set, 
a sequence of iterations $T_{d_n}T_{d_{n-1}}\cdots T_{d_1} (x)$  will have a bounded norm eventually for any $x\in S$. The main obstacle are unimodular eigenvalues of $M$, in particular those with different algebraic and geometric multiplicity.

Our  proof uses the real  Jordan form $J$ of the matrix  $M $.
If  $M= P^{-1}JP$,  

\begin{equation}\label{prevod}{\text{ the equality  \   $z=  \sum_{k\in I}  M^kd_k $ \ \    gives  \ \     $ Pz =\sum_{k\in I} 
J^{k}Pd_k $}}.\end{equation}      

Instead of looking for a finite digit set from the lattice $\Z^m$ which is  suitable for the matrix $M$ and representation of integer vectors,  we look for   a digit set  from the lattice $P\Z^m$, suitable for $J$ and we represent lattice points of   $P\Z^m$.  The matrix $P$ transforming $M$ to  its Jordan form  is not determined uniquely. We choose it carefully to approximate in some sense   the lattice $\Z^m$.    For this purpose,  we introduce the perturbation set 
$$\mathcal{E} = \{ \varepsilon  \in \R^m :  \|\varepsilon \|_\infty <\tfrac13\}.$$
In the article we work with two norms of $\R^m$: 

$\| .\|_\infty$  denotes the  norm   defined by $\|(x_1, x_2,  \ldots, x_m)\|_\infty = \max_{i}|x_i|$ . 

 $\| .\|_2$  denotes the Euclidean norm, i.e. $\|(x_1, x_2,  \ldots, x_m)\|_2 = \sqrt{\sum_{j=1}^m|x_j|^2}.$

\begin{defi}\label{close} We say that a lattice $L \subset \R^m$ is close to the lattice $\Z^m$, if   \ $\Z^m \subset  L +\mathcal{E}$. \end{defi}

For every $M\in\R^{m\times m}$, there exists a non-singular $P\in\R^{m\times m}$ such that $PMP^{-1}M= J ={\bigoplus_{k}J_k}$ where $J_k$ is a real Jordan block.   Let us recall that  the real Jordan block to $\lambda \in \C$  is a matrix in the form 
$$
\left(\begin{smallmatrix}
R&I&& \\[-7pt]
&R&\ddots& \\\vspace*{-1mm}
&&\ddots&I \\[2mm]
&&&R
\end{smallmatrix}\right)\quad
\qquad \text{with} \ 
R=\begin{cases}
             \ (\lambda) & \mbox{ if } \lambda\in\R, \\[2mm]
             \left(\begin{smallmatrix}
               a & b \\
               -b & a
             \end{smallmatrix}\right) &\mbox{ if }  \lambda=a+ib \in\mathbb{C}  \setminus \mathbb{R}
           \end{cases}
$$
Note that $I$ is a unit matrix of order 1 or 2, according to the order of $R$.

The  vector space $\R^m $ can be decomposed  into  invariant subspaces $\R^m = V_e\oplus V_u \oplus V_c$ of the matrix  $J$  such that  $J$ restricted to $V_e$ is expansive,
 $J$ restricted to $V_c$ is contractive and $J$ restricted to $V_u$ is orthogonal.  In other words,  all eigenvalues of  $J$ restricted to $V_e$ are in modulus $>1$,  all eigenvalues of  $J$ restricted to $V_u$ are in modulus $=1$,  and  all eigenvalues of  $J$ restricted to $V_c$ are in modulus $<1$.  
The subspaces $V_e, V_u$ and $V_c$ may be trivial.   Note that  $\dim V_e  + \dim V_u > 0$ as $M$ is an integer non-singular matrix.   
 Any $x \in \R^m$   can be uniquely written as $x = x_e+x_u+x_c$ with $x_e \in V_e$,  $x_u \in V_u$ and $x_c\in V_c$.   We will denote $x_e =\pi_e(x)$, $x_u =\pi_u(x)$,  and $x_c =\pi_c(x)$,

\bigskip

The proof of the main theorem  will be a consequence of the following proposition.

\begin{prop}\label{skoroHotovo} Let $J\in \R^{m\times m}$ be a non-singular  matrix in the real Jordan form     and  $\pi_e: \R^m\to V_e$  be the projection  into the expansive  invariant subspace $V_e$ of  $J$. Let  $L \subset \R^m$ be a lattice close to  $\Z^m$. Then   there exists  a finite set   $\mathcal{D} \subset L$  and a constant $C$  such that

\medskip
\noindent  for any
 $x \in \R^m$  with   $\|\pi_e(x)\|_\infty \leq 1$   and for any sufficiently large $N\in \N$ we  can write  
\begin{equation}\label{Nzmensi}  x = J^{-1}d_1 + J^{-2}d_2  + \cdots + J^{-N}d_N + J^{-N}y,\  \text{ \ for some }      d_1,d_2, \ldots, d_N \in \mathcal{D}\,    \end{equation}
$$ \text{and}    \ \ y \in \R^m, \ \ \text{with} \ \  \|y\|_\infty \leq C\,.$$
\end{prop}

\begin{proof}[Proof of Theorem \ref{main}]
We  use the statement and notation of Proposition  \ref{skoroHotovo}.     Let  $M =P^{-1}JP$, where $J$ is  the real Jordan canonical form of $M$. The  matrix $P$ is not given uniquely. For example,  any matrix $\alpha P$ with a non-zero $\alpha$ transforms $M$ to its Jordan form, as well. Therefore we can  assume without loss of generality that   the lattice $L=P\Z^m$ is close to the lattice $\Z^m$ in the sense of Definition \ref{close}.   As $P^{-1}JP=M$, we get $JP\Z^m =  PM\Z^m \subset P\Z^m$, i.e., $J$ maps the lattice  $L$ into $L$. 
For such   $J$ and $L$, we find by  Proposition    \ref{skoroHotovo}, the digit set $\mathcal{D} \subset L$.   

   Due to  \eqref{prevod},   we have to show that  there exists a finite  digit set $\mathcal{A}\subset L$  such that  
\begin{equation}\label{jinak} \text{ $z \in  L\qquad \Longrightarrow   \quad    \ \ 
z= \sum_{k\in I} 
J^{k}d_k $,  \quad where  $I\subset \Z$ is finite and  $d_k \in \mathcal{A}$  \ for each $k\in I$.}\end{equation}
 \noindent  Define  $$\mathcal{A} := \mathcal{D} + \mathcal{B}, \quad \text{  where }\  
 \mathcal{B} := \{ d \in L :   \|d \|_\infty  \leq C\}.$$ 

 Obviously,   $\mathcal{B}$ is finite   and  $ 0 \in \mathcal{B}$. Therefore $\mathcal{A}$ is  finite  and  $\mathcal{D}\subset \mathcal{A} \subset L$.    Let us show that our  choice  of 
the digit set  $\mathcal{A}$  has the property stated in \eqref{jinak}.

Let    $z \in  L $.     Since $V_e$ is a contractive subspace of the matrix $J^{-1}$, there exists  $j \in \N$ such that   $\|\pi_e(J^{-j}z)\|_\infty < 1$. Applying  \eqref{Nzmensi} to $x =J^{-j}z$ with  $N\geq j$,    we find $d_1, \ldots, d_n \in \mathcal{D}\subset \mathcal{A}$ and $y$ with  $\|y\|_\infty \leq C$  such that 
\begin{equation}\label{cele}J^{-j}z =   J^{-1}d_1 + J^{-2}d_2  + \cdots + J^{-N}d_n + J^{-N}y. \end{equation}

Multiplying \eqref{cele} by $J^N$,  we deduce  $y +d_N = J^{N-j}z - \sum_{k=1}^NJ^{N-k}d_k $.
 As  $J$ maps the lattice  $L$ into $L$,  we have  $y\in L$ and $y \in   \mathcal{B}$.   Obviously,  $y+d_n \in \mathcal{A}$.      Altogether,  $z = \sum_{k=1}^{j-1}J_{-k}d_k + M^{j-N}(d_N+y)$,   and all the coefficients $d_1d_2, \ldots, d_{N-1}$ and $d_{N}+y$ belong to $\mathcal{A}$. 

\end{proof}

\section{Proof of Proposition \ref{skoroHotovo}}\label{sectionProposition}
The proof of Proposition \ref{skoroHotovo} will be done in following way. To the matrix $J$ in the real Jordan form, we  find a  finite digit set $\mathcal D\subset L$, such that the iterations of the transformations $Jx-d$ for some $d\in\mathcal D$, have small norm eventually.

First  we find a digit set  in $\Z^m$ and 
then we  replace  each integer digit by a close element   from  the lattice $L$.   
When working with  the matrix $J$, the  integer lattice has an important advantage. It can be decomposed  into the direct sum  $\Z^m = \oplus_{j} \Z^{m_j}$   where each  lattice   $ \Z^{m_j}$ is contained in an invariant subspace of a real  Jordan block.  Therefore, one can treat each Jordan block separately.

\subsection{A real Jordan block to an eigenvalue  $\lambda$ on the unit circle.}

Let us assume first that  $\lambda$ is not real. The real Jordan block to such a $\lambda$ has an even size, say $2\ell$,   and the form 
\begin{equation}\label{rotace}
J=\left(\begin{smallmatrix}
R&I&& \\[-7pt]
&R&\ddots& \\\vspace*{-1mm}
&&\ddots&I \\[2mm]
&&&R
\end{smallmatrix}\right) \in \mathbb{R}^{2\ell \times 2\ell}\quad
\qquad \text{with } \ 
R= \left(\begin{smallmatrix}
               \cos \varphi& - \sin\varphi  \\
              \sin\varphi& \cos \varphi
             \end{smallmatrix}\right)   \ 
\ \  \text{and } \ 
I= \left(\begin{smallmatrix}
               1& 0  \\
              0& 1
             \end{smallmatrix}\right).
\end{equation}

\begin{defi}\label{defIndex} Let  $K_1, K_2, \ldots, K_\ell$ be positive constants and $J$ the matrix given in \eqref{rotace}.  We define the index 
${\rm ind}: \mathbb{R}^{2\ell} \mapsto \{0,1,\cdots, \ell\}$  in the following way.  For  $x = (x_1,x_2, \ldots, x_{2\ell - 1}, x_{2\ell} )^T $, we put   
\begin{itemize}
\item 
  ${\rm ind}(x) = 0$,\  if \   $  \|(x_{2i-1}, x_{2i})\|_2 < K_i$   for all $i = 1,2, \ldots, \ell$;  

\item  
 otherwise,    \  ${\rm ind}(x) = $ the maximal index $j$   such that $  \|(x_{2j-1}, x_{2j})\|_2 \geq  K_j$. 

\end{itemize}  
 \end{defi}

First we  show two simple claims. 

\begin{claim}\label{claim1} Given $q \in \N,$  denote $\mathcal{B} = \{ (0, 0)^T, (3q, 0)^T,   (0,3q)^T, (-3q, 0)^T,   (0,-3q)^T \}\subset \R^2$.   Then for each $z \in \R^2$ there exists $b \in \mathcal{B}$ such that  
$$
\|z - b\|_2 \leq 6q \qquad \text{or} \qquad \|z - b\|_2 \leq \|z \|_2 - q. \
$$
\end{claim}
\begin{proof}  Because of the  symmetry of the digit set  $\mathcal{B}$,  it is enough to consider $z=(z_1,z_2)^T \in \R^2$ with $z_1\geq z_2\geq 1$. 

\medskip

If $z_1\geq 3q$, put $b =(3q, 0)^T$. Then $\|z - b\|_2 = \sqrt{(z_1 - 3q)^2 + z_2^2}$. The inequality  $ \|z - b\|_2 \leq \|z \|_2 - q$  we want to show is equivalent to  $4q + \sqrt{z_1^2+z_2^2} < 3z_1$. The last inequality can be easily checked, since $4q < (3-\sqrt{2})z_1$  and  $\sqrt{z_1^2+z_2^2} \leq  \sqrt{2}z_1$.   
\medskip

If $z_1<3q$, put $b =(0,  0)^T$. Then  $\|z - b\|_2 = \sqrt{z_1^2 + z_2^2}\leq \sqrt{(3q)^2 + (3q)^2} \leq 6q$. 
\end{proof}
\begin{claim}\label{claim2} Let a constant $c_1\geq 0$ and $ R= \left(\begin{smallmatrix}
               \cos \varphi& - \sin\varphi  \\
              \sin\varphi& \cos \varphi
             \end{smallmatrix}\right)   \in \R^{2\times 2}$ be given.
             Then there exist  a constant $c_2>0$ and $\mathcal{B}  \subset \Z^2$ with $\# \mathcal{B} = 5$ such that  for any $u, v \in \R^2$, $ \|v\|_2 \leq c_1,$ there exists $b \in \mathcal{B}$ such that  
$$\|Ru + v - b\|_2 < c_2 -\tfrac12 \qquad \text{or}\qquad \|Ru + v - b\|_2 < \|u\|_2 - 1.$$
\end{claim}
\begin{proof} Find  $q \in \N, q\geq c_1+1$.  For this  $q$,  Claim \ref{claim1} gives us  $\mathcal{B}$. Denote $c_2 = 6q+\tfrac12$. 
Any rotation $R$ preserves the euclidean norm, i.e.,  $\|Ru\|_2  = \|u \|_2$.   According to Claim \ref{claim1}, for each vector $z =Ru + v$ we find $b \in \mathcal{B}$ such that  
$$\|Ru + v - b\|_2 < 6q = c_2-\tfrac12  \quad \text{or}\quad \|Ru + v - b\|_2 < \|Ru + v \|_2 - q \leq  \|u \|_2 + \|v \|_2 - q \leq   \|u \|_2 + \underbrace {c_1 - q}_{\leq -1}\,.$$

\end{proof}

\begin{lem}\label{unit}  Let $J$ be the matrix given in \eqref{rotace}. There exist a finite digit set $\mathcal{D} \subset \Z^{2\ell}$ and constants $K_1,\dots,K_\ell$, such that the function ${\rm ind}$ defined by $K_1,\dots,K_\ell$ has the following property. 

For each $x \in  \R^{2\ell}$ there exists a digit  $ d \in \mathcal{D}$  such that for all  $\varepsilon  \in  \mathcal{E}  $,   the vector $y= Jx-d+ \varepsilon$  satisfies
\medskip

\begin{itemize}

\item $ {\rm ind}(y) \leq  {\rm ind}(x) $;
\medskip

\item   If    \ $ {\rm ind}(y) = {\rm ind}(x) = j $, then $  \|(y_{2j-1},y_{2j})\|_2 \leq   \|(x_{2j-1}, x_{2j})\|_2 -\tfrac12$.
 
\end{itemize}
\end{lem}

\medskip

\begin{proof}

Consider the matrix   $J \in \R^{2\ell \times 2\ell}$  given in \eqref{rotace}. Using $\ell$  times Claim \ref{claim2} , we will construct $\ell$ constants $K_1, \ldots, K_\ell$ to determine  the function ${\rm ind}$:   

by $K_\ell$  and $\mathcal{B}_\ell$   we denote the constant $c_2$ and the digit set found by Claim \ref{claim2} for $c_1 = 0$;   

by $K_{\ell-1}$  and $\mathcal{B}_{\ell-1}$   we denote the constant $c_2$ and the digit set found by Claim \ref{claim2}  for $c_1 = K_\ell$;  

by $K_{\ell-2}$  and $\mathcal{B}_{\ell-2}$   we denote the constant $c_2$ and the digit set found by Claim \ref{claim2}  for $c_1 = K_{\ell-1}$;

\noindent etc. 

\medskip

The  digit set $\mathcal{D}$  is defined by
\begin{equation}\label{abecedaZB}
d\in {\mathcal{D}} \qquad \text{ if and only if } \qquad 
d=\left(\begin{smallmatrix}
b^{(1)}\\\vspace*{-2mm}
b^{(2)} \\\vspace*{-1mm}
\vdots\\[2mm]
b^{(\ell)} 
\end{smallmatrix}\right) \in \mathbb{Z}^{2\ell}, \ \ 
\text{where} \ \ b^{(k)}  \in \mathcal{B}_k, \text{ for }  k=1, 2, \ldots, \ell.
\end{equation}

We show that the function ${\rm ind}$ defined by the constants    $K_1, \ldots, K_\ell$  and the digit set $\mathcal{D}$ have the property declared in the statement of the lemma.

Let $x \in \R^{2\ell}$. 
Then $$\left(\!\!\!\begin{array}{l}
(Jx)_{2i-1}\\
(Jx)_{2i}
\end{array}\!\!\!\right) = R \left(\!\!\!\begin{array}{l}
x_{2i-1}\\
x_{2i}
\end{array}\!\!\!\right) + \left(\!\!\!\begin{array}{l}
x_{2i+1}\\
x_{2i+2}
\end{array}\!\!\!\right), \ \ \ \text{ for $i = 1,2, \ldots, \ell$, where we put}\ \    \left(\!\!\!\begin{array}{l}
x_{2\ell+1}\\
x_{2\ell+2}
\end{array}\!\!\!\right)  =  \left(\!\!\!\begin{array}{l}
0\\
0
\end{array}\!\!\!\right).$$
To this  $x$, we define $d \in \mathcal{D}$ by determining  all its components $b^{(i)}$. 

\begin{itemize}
\item 
If  $i <   {\rm ind}(x)$,  then we choose $b^{(i)} \in \mathcal{B}_i$ randomly.

\item 

If $i\geq   {\rm ind}(x)$, then $ \|(x_{2i+1}, x_{2i+2})\|_2 < K_{i+1}$ and  $b^{(i)} \in \mathcal{B}_i$ is determined by Claim \ref{claim2}   for $ u= (x_{2i-1}, x_{2i})^T$  and   $ v= (x_{2i+1}, x_{2i+2})^T$. Recall that  we used Claim \ref{claim2}  to find $K_{i}: = c_2$   for  $c_1 := K_{i+1}$.   
Therefore, 
\begin{equation}\label{oba} \left\|\left(\!\!\!\begin{array}{l}
(Jx)_{2i-1}\\
(Jx)_{2i}
\end{array}\!\!\!\right)  - b^{(i)} \right\|_2 <  K_{i} -\tfrac12 \quad \text{or} \quad  \left\|\left(\!\!\!\begin{array}{l}
(Jx)_{2i-1}\\
(Jx)_{2i}
\end{array}\!\!\!\right)  - b^{(i)} \right\|_2 <  \left\|\left(\!\!\!\begin{array}{l}
x_{2i-1}\\
x_{2i}
\end{array}\!\!\!\right) \right\|_2 - 1.  
\end{equation}

\end{itemize}
For the digit $d$ we have described above and an  $\varepsilon  \in \mathcal{E}$,  we focus on  $y =Jx -d+\varepsilon$. 
 Let us realize that the inequality $\|\varepsilon\|_\infty < \tfrac13$ means  $|\varepsilon_j|\leq \tfrac13$ for each coordinate $j$ and  thus 
$\|(\varepsilon_{2i-1}, \varepsilon_{2i})\|_2 \leq  \tfrac12$. 

\medskip

For each $i> {\rm ind}(x)$,  we have  $ \|(x_{2i-1}, x_{2i})\|_2 < K_{i}$,   and  thus  both inequalities in \eqref{oba}  imply 
$$
 \left\|\left(\!\!\!\begin{array}{l}
(Jx)_{2i-1}\\
(Jx)_{2i}
\end{array}\!\!\!\right)  - b^{(i)} +     \left(\!\!\!\begin{array}{l}
\varepsilon_{2i-1}\\
\varepsilon_{2i}
\end{array}\!\!\!\right)  \right\|_2 < K_{i}.
$$
Consequently,  ${\rm ind}(y) \leq  {\rm ind}(x)$. 

\medskip 

 If  ${\rm ind}(y) =  {\rm ind}(x) =:j$,   then necessarily   
$\|(y_{2j-1}, y_{2j})\|_2 =  \left\|\left(\!\!\!\begin{array}{l}
(Jx)_{2j-1}\\
(Jx)_{2j}
\end{array}\!\!\!\right)  - b^{(j)} +   \left(\!\!\!\begin{array}{l}
\varepsilon_{2j-1}\\
\varepsilon_{2j}
\end{array}\!\!\!\right)  \right\|_2 > K_{j}$  and the left most inequality in  \eqref{oba} cannot hold true. The  validity of the right most 
inequality implies $\|(y_{2i-1}, y_{2i})\|_2  \leq \|(x_{2i-1}, x_{2i})\|_2 - \tfrac12$.

\end{proof}

In case of real $\lambda$, that is, $\lambda = \pm1,$ the approach is analogous, only easier. The digit set 
$\{(a_1,\dots,a_m)^T : a_i\in\{-1,0,1\}\}$ now suffices for the analogy Lemma~\ref{unit}.

\begin{coro}\label{unitCoro} Let $J=J_m(\lambda)$ with $|\lambda|=1$. Then there exists $\mathcal D\subset\Z^m$ finite with the following property. For each $x \in \R^{m}$ there exists  a sequence  $(d_{i+1})_{i\in \N}$ of digits from $\mathcal{D} $  such that  for any  sequence  $(\varepsilon_{i+1})_{i\in \N}$ from  the perturbation set $\mathcal{E}$, the sequence   defined recursively 

\medskip 

\centerline{$x^{(0)} = x$  \ \ and \ \ \ $x^{(n+1)}:= Jx^{(n)} - d_{n+1}+\varepsilon_{n+1}$ for $n \in \N$}

\medskip 

\noindent  satisfies  ${\rm ind}(x^{(N)})  = 0$ for all sufficiently large $N \in \N$.   In particular,  there exists a constant $C$ such that $\|x^{(N)}\|_\infty  < C$  for all sufficiently large $N\in \N$.

\end{coro}

\subsection{A real Jordan block to an eigenvalue  $\lambda$ strictly inside the unit circle }
\begin{lem}\label{contract} Let $J\in \R^{m\times m}$ be the real Jordan block to $\lambda$ of modulus $<1$ and     $\mathcal{E}$ be the perturbation set.   Then there exist a norm $\|.\|_c$ of  $\R^{m}$ and  a constant $\gamma$ such that

\medskip 
 for any  $x \in \R^m$  and  any $\varepsilon \in \mathcal{E}$, the vector $y=Jx+\varepsilon$ satisfies: 

\medskip 

\begin{itemize}

\item  If \ $ \| x\|_c   \leq \gamma$, then $\|y\|_c \leq \gamma$;
\medskip

\item   If  \ $ \| x\|_c   \geq \gamma$, then $\|y\|_c \leq \|x\|_c - \tfrac12$. 
 
\end{itemize}
\end{lem}

\begin{proof}  Let us choose $\beta$ such that $1> \beta > |\lambda|$.  By  Theorem 3 from \cite{IsaacsonKeller}, there exists a norm $ \| x\|_c $  of  $\R^m$  such that    $\|Jz\|_c \leq \beta \|z\|_c$  \  for all $z \in \R^m$. Then   
\medskip

\centerline{$\|Jx+ \varepsilon\|_c \leq \|Jx\|_c + \|\varepsilon\|_c\leq \beta \|x\|_c   +E$ \ \  where $E=\max\{\|\varepsilon\|_c : \varepsilon \in \mathcal{E}\}$.}

\medskip

 \noindent Now,  it is enough to check that 

 $ \| x\|_c   \leq   \frac{1}{1-\beta}\bigl(\tfrac12+E\bigr) $,  implies  $ \beta \|x\|_c   +E \leq   \frac{1}{1-\beta}\bigl(\tfrac12+E\bigr)$

\noindent and 

 $ \| x\|_c   \geq   \frac{1}{1-\beta}\bigl(\tfrac12+E\bigr) $,  implies  $ \beta \|x\|_c   +E \leq \|x\|_c   - \tfrac12 $. 
\medskip

\noindent Therefore, we put $\gamma = \frac{1}{1-\beta}\bigl(\tfrac12+E\bigr)$. 

\end{proof}

\begin{coro}\label{contractCoro} Let $J\in \R^{m\times m}$ be the real Jordan block to $\lambda$ of modulus $<1$. Then there exists a constant $\gamma$  and a norm $ \| x\|_c $  of  $\R^m$  such that  for each $x \in \R^{m}$  and   for any  sequence  $(\varepsilon_{i+1})_{i\in\N}$ from  $\mathcal{E}$, the sequence   defined recursively 
$$x^{(0)} = x  \ \ and \ \ \ x^{(n+1)}:= Jx^{(n)} +\varepsilon_{n+1} \quad\text{ for }\quad n \in \N$$
satisfies   $\|x^{(N)}\|_c  < \gamma$  for all sufficiently large $N \in \N$.

\end{coro}

\subsection{A real Jordan block to an eigenvalue  $\lambda$ strictly outside the unit circle }
\begin{lem}\label{expand}   Let $J\in \R^{m\times m}$ be the real Jordan block to $\lambda$ of modulus $>1$ and     $\mathcal{E}$ be the perturbation set.   Then there exists a digit set $\mathcal{D} \subset \Z^m$ with the property: 

\medskip
for each $x \in  \R^{m}$    with  $\|x\|_\infty \leq 1$  there exists a digit  $ d \in \mathcal{D}$  such that
 for all  $\varepsilon  \in  \mathcal{E}  $,   the vector $y= Jx-d+ \varepsilon$  satisfies  $\|y\|_\infty\leq 1$.  

\end{lem}
\begin{proof} Since   $\|Jx\|_\infty \leq (2|\lambda| + 1)\|x\|_\infty $ for each $x\in  \R^{m}$, we put 
$\mathcal{D} = \{ d\in \Z^m : \|d\|_\infty \leq  2|\lambda| + 2 + E\}$,  where $E=\max\{\|\varepsilon\|_\infty : \varepsilon \in \mathcal{E}\}$.

\end{proof}

\begin{coro}\label{expandCoro}  Let $J\in \R^{m\times m}$ be the real Jordan block to $\lambda$ of modulus $>1$,  and $\mathcal{D}$   and   $\mathcal{E}$  be  as in Lemma \ref{expand}. Then for each $x \in \R^{m}$ there exists  a sequence  $(d_{i+1})_{i\in \N}$ of digits from $\mathcal{D} $  such that  for any  sequence  $(\varepsilon_{i+1})_{i\in \N}$ from  $\mathcal{E}$ the sequence   defined recursively 

\medskip 

\centerline{$x^{(0)} = x$  \ \ and \ \ \ $x^{(n+1)}:= Jx^{(n)} - d_{n+1}+\varepsilon_{n+1}$  \ \ for $n \in \N$,}

\medskip 

\noindent  satisfies  $\|x^{(N)}\|_\infty  < 1$  for all sufficiently large $N \in \N$.   

\end{coro}

\subsection{Completion of proof of Proposition \ref{skoroHotovo}}

Now we combine properties of all types of the real Jordan  blocks.

\begin{proof}[Proof of Proposition \ref{skoroHotovo}]
Let $J_1, J_2, \ldots , J_s$ be the real Jordan blocks of the size $m_1, \ldots, m_s$ respectively,  such that 
$J=J_1\oplus J_2 \oplus \cdots \oplus J_s$. 

To each block $J_k \in \R^{m_k}$ corresponding to  $\lambda$ with $|\lambda|\geq 1$ we find a constant $C_k$ and  a finite digit set $\mathcal{D}_k\subset \Z^{m_k}$  with the properties described in   Corollaries \ref{expandCoro} and \ref{unitCoro}.

 To  each block $J_k$ corresponding to  $\lambda$ with $|\lambda|< 1$, we  assign the constant $\gamma_k$  found in Corollary \ref{contractCoro}  for the norm $\|.\|_c$.    As two norms on a finite dimensional space are equivalent, there exists a constant  $C_k$ such that $\|x\|_c < \gamma_k$ implies  $\|x\|_\infty < C_k$.   We put $\mathcal{D}_k=\{0\} \subset \R^{m_k}$.

We use the digit sets $\mathcal{D}_k$  for construction of a new digit set $\widetilde{\mathcal{D}}\subset   \Z^m$
\begin{equation}\label{abeceda}
\tilde{d}\in \widetilde{\mathcal{D}} \qquad \text{ if and only if } \qquad 
\tilde{d}=\left(\begin{smallmatrix}
d^{(1)}\\\vspace*{-2mm}
d^{(2)} \\\vspace*{-1mm}
\vdots\\[2mm]
d^{(s)} 
\end{smallmatrix}\right) \in \mathbb{Z}^{m}, \ \ 
\text{where} \ \ d^{(k)}  \in \mathcal{D}_k, \text{ for }  k=1, 2, \ldots, s.
\end{equation}
Put $C =\max_k{C_k}$. In this notation, we get the following.

\medskip 
\noindent {\bf Claim A:}
 for each      $x \in \R^m$  with   $\|\pi_e(x)\|_\infty \leq 1$  there exists a sequence  $(\tilde{d}_i)_{i\in \N}$ of digits from $\widetilde{\mathcal{D}} $  such that  for any  sequence  $(\varepsilon_i)_{i\in \N}$ from  $\mathcal{E}$ the sequence   defined recursively 
\begin{equation}\label{jesteIntegerAlphabet}
\text{$x^{(0)} = x$  \ \ and \ \ \ $x^{(n+1)}:= Jx^{(n)} - \tilde{d}_{n+1}+\varepsilon_{n+1}$  \ \ for $n \in \N$,} 
\end{equation}
\noindent  satisfies  $\|x^{(N)}\|_\infty  < C$  for all sufficiently large $N \in \N$.  

\medskip  

Since $L$ is close to  $ \mathbb{Z}^{m}$,   we can assign  to  each $\tilde{d} \in \widetilde{\mathcal{D}}$ a lattice point  $ d\in L$ such that $d \in \tilde{d} +\mathcal{E}$. We obtain a new digit set  $\mathcal{D} \subset L$ and the size of  $\mathcal{D}$ does not exceed the size of   $\widetilde{\mathcal{D}}$.   We rewrite Equation \eqref{jesteIntegerAlphabet} for the specific choice of the sequence   $(\varepsilon_i)_{i\in \N}$, namely for $\varepsilon_i : = {\tilde{d}}_i - d_i \in \mathcal{E}$.   We get a modification of the previous claim

\medskip 
\noindent {\bf Claim B:}
 for each      $x \in \R^m$  with   $\|\pi_e(x)\|_\infty \leq 1$  there exists a sequence  $({d}_i)_{i\in \N}$ of digits from ${\mathcal{D}} $  such that   the sequence   defined recursively 
$$
\text{$x^{(0)} = x$  \ \ and \ \ \ $x^{(n+1)}:= Jx^{(n)} - {d_{n+1}}$  \ \ for $n \in \N$,} 
$$
\noindent  satisfies  $\|x^{(N)}\|_\infty  < C$  for all sufficiently large $N \in \N$.  

\medskip
Multiplying the equality  $x^{(n+1)}= Jx^{(n)} - {d_{n+1}}$  by $J^{-n -1}$  and summing up for $n= 0,1,\ldots, N-1$ we get 
$$\sum_{n=0}^{N-1} J^{-n-1}x^{(n+1)} = \sum_{n=0}^{N-1} J^{-n}x^{(n)} - \sum_{n=0}^{N-1} J^{-n-1}d_{n+1}   \quad \text{and thus}   \quad \sum_{n=1}^{N} J^{-n}x^{(n)} = \sum_{n=0}^{N-1} J^{-n}x^{(n)} - \sum_{n=1}^{N} J^{-n}d_{n}\,.$$ 
It implies $ J^{-N}x^{(N)} =  x^{(0)} - \sum_{n=1}^{N} J^{-n}d_{n}$  with     $\|x^{(N)}\|_\infty  < C$, as required in the statement of the proposition.

\end{proof}

\section{${\rm Fin}_{\mathcal{D}}(M)$ \ \  versus  \ \  $\bigcup\limits_{k\in \mathbb{N}} \frac{1}{\Delta^k}\mathbb{Z}^{m}$}
In general, we have that ${\rm Fin}_{\mathcal D}(M)\subseteq\bigcup\limits_{k\in \mathbb{N}} \frac{1}{\Delta^k}\mathbb{Z}^{m}$. In this section we describe when the inclusion turns into an equality.
The inverse to a matrix $M \in \mathbb{Z}^{m\times m}$ with the determinant  $\Delta \neq 0$  belongs to  $\frac{1}{\Delta}\mathbb{Z}^{m\times m}$. Therefore, 
the definition of the set  ${\rm Fin}_{\mathcal{D}}(M)$  directly implies the following properties. 

\begin{lem}\label{list}  Let $M \in \mathbb{Z}^{m\times m}$ with  $\det M = \Delta \neq 0$ and $\mathcal{D} \subset \mathbb{Z}^m$, $\mathcal{D}$ finite.  Then 
\begin{enumerate}  
\item  $M^k\Bigl({\rm Fin}_{\mathcal{D}}(M)\Bigr) =  {\rm Fin}_{\mathcal{D}}(M)$ for each $k \in \Z$\,.

\item   ${\rm Fin}_{\mathcal{D}}(M) \subseteq  \bigcup\limits_{k\in \mathbb{N}} \frac{1}{\Delta^k}\mathbb{Z}^{m} $\,.

\item   If $ \mathbb{Z}^m \subset {\rm Fin}_{\mathcal{D}}(M)$,  then  ${\rm Fin}_{\mathcal{D}}(M)$ is closed under addition and subtraction. 
\end{enumerate}
\end{lem}

Any  classical $b$-ary numeration system, with  base $b \in \N, b\geq 2$ and the canonical digit set $\mathcal{D} = \{0,1,\ldots, b-1\}$  can be considered as a matrix system with $M=b \in \Z^{1\times 1}$ and the determinant  $\Delta =b$.   The canonical digit set  allows to represent only non-negative numbers and obviously ${\rm Fin}_{\mathcal{D}}(M) =  \bigcup\limits_{k\in \mathbb{N}} \frac{1}{b^k}\mathbb{N} $.

In the next auxiliary claim we denote by $e_i \in \Z^m$ the vector $(\delta_{i1}, \delta_{i2}, \ldots, \delta_{im})^T$, where $\delta_{ij}$ is the Kronecker symbol, i.e. $\delta_{ij} = 0$, if $i\neq j$ and $\delta_{ii} = 1$.

\begin{claim}\label{C1}  $  {\rm Fin}_{\mathcal{D}}(M) =  \bigcup\limits_{k\in \mathbb{N}} \frac{1}{\Delta^k}\mathbb{Z}^{m}$ \ \  if and only if \ \    for each $i \in \{ 1,2, \ldots, m\}$ there  exists  $\ell_i \in \N$ such that  $\frac{1}{\Delta} {e}_i \in  M^{-\ell_i}(\Z^m)\,.$ 
\end{claim}
\begin{proof}  $(\Leftarrow)$   Due to Item (2) of Lemma \ref{list}, it is enough to show that  $ \frac{1}{\Delta^k}\mathbb{Z}^{m} \subset {\rm Fin}_{\mathcal{D}}(M) $ for each $k \in \N$. We show it by induction on $k \in \N$.  

The assumption  $\frac{1}{\Delta} {e}_i \in  M^{-\ell_i}(\Z^m)$ gives $\tfrac{1}{\Delta}\Z^m\subset \sum_{i=1}^mM^{-\ell_i}(\Z^m)$.  Multiplying  this inclusion by $\tfrac{1}{\Delta^k}$ and 
 using the induction hypothesis we obtain 
$$
\tfrac{1}{\Delta^{k+1}} \Z^m  \subset  \sum_{i=1}^mM^{-\ell_i}\bigl(\tfrac{1}{\Delta^{k}} \Z^m\bigr) \subset \sum_{i=1}^mM^{-\ell_i}\bigl({\rm Fin}_\mathcal{D}(M)\bigr) =  {\rm Fin}_\mathcal{D}(M)\,.
$$
The last equality follows from Items (1) and (3) of Lemma \ref{list}.

$(\Rightarrow)$   This implication is obvious.  
\end{proof}

\begin{thm}\label{rovnost}  Let $M \in \mathbb{Z}^{m\times m}$, $\det M  = \Delta \neq 0$.   Let  $\mathcal{D} \subset \Z^m$ be a finite digit set such that   $\Z^m \subset  
  {\rm Fin}_{\mathcal{D}}(M)$.    Then 
$$  {\rm Fin}_{\mathcal{D}}(M) =  \bigcup\limits_{k\in \mathbb{N}} \frac{1}{\Delta^k}\mathbb{Z}^{m}  \qquad \Longleftrightarrow
 \qquad  \exists \ell \in \N \text{  such that }  M^\ell = \Theta \mod \Delta\,.$$
In particular,  if  $\det M = \pm 1$, then $ \mathbb{Z}^{m} =  {\rm Fin}_{\mathcal{D}}(M)$\,.

\end{thm}
\begin{proof} By Claim \ref{C1} we have to study when $\frac{1}{\Delta}e_i \in M^{-\ell_i}(\Z^m)$ for some $\ell_i \in \N$.  It can be rewritten into  $M^{\ell_i} e_i \in \Delta \Z^m$, or equivalently, $ M^{\ell_i} e_i = 0\mod \Delta$.     We  denote the   $i^{th}$ column of a matrix $A$ by $A_{\bullet i}$.  Obviously,  $M^{\ell_i}e_i = (M^{\ell_i})_{\bullet i}$. 

 Let us realize that  if $(M^{\ell_i})_{\bullet i}= 0 \mod \Delta$,  then   $(M^{\ell_i})_{\bullet i}=0 \mod \Delta$ for all $\ell\in \N, \ell>\ell_i$. Indeed,  $(M^{\ell_i})_{\bullet i} = M^{\ell -\ell_i} (M^{\ell_i})_{\bullet i} =  M^{\ell -\ell_i} 0 \mod \Delta $.  Therefore,   we can look for a common exponent $\ell$ for all columns of the matrix $M$.  
\end{proof}

\bigskip
\begin{example} Let $M_1 = \left( \begin{array}{rr} 2&1\\2&2 \end{array} \right) $  and  $M_2 = \left( \begin{array}{rr} 1&2\\2&2 \end{array} \right) $. 
Then $\det M_1 = 2$ and $\det M_2 = -2$.  
$$M_1 = \left( \begin{array}{rr} 0&1\\0&0 \end{array} \right) \mod \Delta,   \qquad   M_1^2 = \Theta \mod \Delta \qquad \text{and thus \ } {\rm Fin}_{\mathcal{D}}(M) =  \bigcup\limits_{k\in \mathbb{N}} \frac{1}{2^k}\mathbb{Z}^{2}   $$  
for a suitable   $\mathcal{D} \subset \Z^2$. 

On the other hand,  $$M_2 = \left( \begin{array}{rr} 1&0\\0&0 \end{array} \right) \mod \Delta,   \qquad   \text{and thus \ } \ \  M_2^\ell  \neq  \Theta \mod \Delta \  \text{ for all $\ell \in \N$\ }  \,. $$  
In particular,  $\Bigr(\!\!\! \begin{array}{r} \tfrac12\\0 \end{array}\!\!\! \Bigl) \notin {\rm Fin}_{\mathcal{D}}(M_2)$ for any choice $\mathcal{D} \subset \Z^2$.  

\end{example}

\section{Comments}
The main idea of this paper comes from~\cite{VavraIsrael} concerning number systems with algebraic base without any restrictions on its Galois conjugates. We do not give an explicit bound on the size of the digit set, even though our proof is a constructive one. For matrices without eigenvalues on the unit circle, the method used in \cite{FrPeSv},\cite{FrHePeSv} can be used to obtain explicit digit sets, as well as arithmetic algorithms. 



To find a digit set $\mathcal{D} \subset \Z^m$ of the minimal cardinality may be a hard problem. It is known that if $M$ is expanding, then we need a digit set of size at least  $|\det M|$ to be able to represent integer vectors in the form $\sum_{k=0}^nM^kd_k$. In our case it seems that an appropriate lower bound would be around the product of absolute values of the eigenvalues outside the unit circle.

In any case, the digit set needs to have at least one digit beside zero. Two digits might be already enough for some matrices of dimension two or more, as is illustrated by the following examples.

The first example is a case of a matrix having only one eigenvalue outside the unit circle. It can be seen from the proof of Proposition~\ref{skoroHotovo} that this dominant eigenvalue will determine the size of the digit set. We will, however, utilize the connection between matrix numeration systems and those with real base.

\begin{example}   Let $T = \left(\!\!\! \begin{array}{rrr} 0&0&-1\\1&0&1\\0&1&-1 \end{array} \!\!\!\right) $. 
As $\det T =-1$, we have $  {\rm Fin}_{\mathcal{D}}(T)  = \Z^3$ for a suitable digit set guaranteed by Theorem \ref{main}. The characteristic polynomial of $T$ is  $x^3+x^2-x+1$. $T$ has one negative eigenvalue $\beta = -1,839\dots$ and a pair of complex conjugate eigenvalues in  modulus smaller than 1. As shown in \cite{KrStVa},  any element of  the ring  $\Z[\beta]=\{a+b\beta +c\beta^2: a,b,c \in \Z\}$    can be written as 
\begin{equation}\label{tribo} a+b\beta +c\beta^2 = \sum_{i=k}^n a_i\beta^i,  \quad \text{where }  k, n \in \Z, k\leq n  \text{ and \ }  a_i \in \{0,1\} \,.\end{equation}

We use the isomorphism $\psi :  \Z[\beta]  \to \Z^3$ of two additive groups  given by   $\psi(a+b\beta +c\beta^2) = \left(\!\!\!  \begin{array}{r} a\\b\\c \end{array}\!\!\! \right)$. 
Since $\beta  (a+b\beta +c\beta^2 ))=  -c +(a+c)\beta +(b-c)\beta^2$, we have  
$$\psi(\beta (a+b\beta +c\beta^2 ) )=\left(\!\!\!  \begin{array}{c} -c\\ a+c \\b-c \end{array}\!\!\! \right) = \left(\!\!\! \begin{array}{rrr} 0&0&-1\\1&0&1\\0&1&-1 \end{array} \!\!\!\right)\left(\!\!\!  \begin{array}{r} a\\b\\c \end{array}\!\!\! \right)  =T\left(\!\!\!  \begin{array}{r} a\\b\\c \end{array}\!\!\! \right).$$ 
Applying this rule to  Equation \eqref{tribo} we deduce that   any element of $ \Z^3$  can be expressed as  
$$\left(\!\!\!  \begin{array}{r} a\\b\\c \end{array}\!\!\! \right) =  \psi(a+b\beta +c\beta^2)=  \sum_{i\in I}\psi(\beta^i a_i) =  \sum_{i=k}^n T^i \psi( a_i) = \sum_{i=k}^n T^i\left(\!\!\!  \begin{array}{c} a_i\\0\\0 \end{array}\!\!\! \right)\,.$$
In other words, $  {\rm Fin}_{\mathcal{A}}(T)  = \Z^3$ already for the small digit set 
$
\mathcal{A} =  \{ (0,0,0)^T, (1,0,0)^T \}
$.

\end{example}

The second example is a rotation matrix in $\R^2.$ Notice that in this case the link to real/complex base numeration systems is missing, for there is no eigenvalue of modulus $>1.$ Moreover, the actual minimal digit set is significantly smaller that the one ensured by the the proof of Proposition~\ref{skoroHotovo}.

\begin{example}
Consider $M = \begin{pmatrix}0&1\\-1&0\end{pmatrix}$ with the digit set $\mathcal D=\{(0,0)^T,(1,0)^T\}.$ 

We have the identities
\[
M^{4k}=\begin{pmatrix}1&0\\0&1\end{pmatrix},\quad M^{4k+1} = \begin{pmatrix}0&1\\-1&0\end{pmatrix},\quad M^{4k+2}=\begin{pmatrix}-1&0\\0&-1\end{pmatrix},\quad M^{4k+3}=\begin{pmatrix}0&-1\\1&0\end{pmatrix}.
\]

Hence 
\[
M^{4k}\begin{pmatrix}1\\0\end{pmatrix} = \begin{pmatrix}1\\0\end{pmatrix},
\quad M^{4k+1}\begin{pmatrix}1\\0\end{pmatrix}=\begin{pmatrix}0\\-1\end{pmatrix},
\quad M^{4k+2}\begin{pmatrix}1\\0\end{pmatrix}=\begin{pmatrix}-1\\0\end{pmatrix},
\quad M^{4k+3}\begin{pmatrix}1\\0\end{pmatrix}\begin{pmatrix}0\\1\end{pmatrix}.
\]
It is now easy to see how the representations are created. For instance, if $a,b>0,$ then
\[
\begin{pmatrix}a\\b\end{pmatrix} = \sum_{k=1}^a M^{4k}\begin{pmatrix}1\\0\end{pmatrix} + \sum_{k=1}^b M^{4k+3}\begin{pmatrix}1\\0\end{pmatrix}.
\]
\end{example}

\section*{Acknowledgements}
E.P.  acknowledges support of the project CZ.02.1.01/0.0/0.0/16\_019/0000778. T.V. was supported by Charles University,  projects PRIMUS/20/SCI/002 and UNCE/SCI/022.


\begin{thebibliography}{11}


 \bibitem{ABBPT05}  S. Akiyama, T. Borb\'ely, H. Brunotte, A. Peth\H{o}, and J. M. Thuswaldner, Generalized
radix representations and dynamical systems. I, Acta Math. Hungar., 108 (2005), pp. 207–238.


\bibitem{EvertseSpol19} J.-H. Evertse, K. Gy\H{o}ry, A. Peth\H{o}, J. Thuswaldner, 
Number systems over general orders, Acta Mathematica Hungarica 159(1) (2019), 187–-205. 


\bibitem{FrHePeSv} Ch. Frougny, P. Heller, E. Pelantová, and M. Svobodová, k-block parallel addition versus 1-block parallel addition in non-standard numeration systems, Theor. Comp. Sci., 543 (2014), 52-67. 

\bibitem{FrPeSv} Ch. Frougny, E. Pelantová, and M. Svobodová. Parallel addition in non-standard numeration systems, Theor. Comp. Sci., 412 (2011), 5714-5727. 


 \bibitem{Gilbert81} W. J. Gilbert, Radix representations of quadratic Fields, J. Math. Anal. Appl., 83 (1981), 264--274.
 
 \bibitem{IsaacsonKeller} E. Isaacson and H. B. Keller,  Analysis of numerical methods, John Wiley\& Sons 1966.
 
 
 \bibitem{JT} J. Jankauskas and J. Thuswaldner, Characterization of rational matrices that admit finite digit representations, Linear Algebra Appl., 557 (2018), 350--358.
 
\bibitem{KataiKovacs80} I. K\'atai and B. Kov\'acs, Kanonische Zahlensysteme in der Theorie der quadratischen algebraischen Zahlen,
Acta Sci. Math. (Szeged), 42 (1980), 99--107.

\bibitem{KataiKovacs81} I. K\'atai and B. Kov\'acs, Canonical number systems in imaginary quadratic fields, Acta Math. Acad. Sci.
Hungar., 37 (1981), 159--164.
 
 
 


\bibitem{KataiSzabo75} I. K\'atai  and  J. Szab\'o, Canonical number systems for complex integers, Acta Sci. Math. (Szeged) 37:3-4(1975), 255--260. 

\bibitem{KovacsBela81} B. Kov\'acs. Canonical number systems in algebraic number fields. Acta Math. Acad. Sci.
Hungar., 37(4):405–407, 1981.

\bibitem{KovacsPetho83} B. Kov\'acs and A. Peth\H{o}. Canonical systems in the ring of integers. Publ. Math. Debrecen,
30(1-2):39–45, 1983.
 
 

\bibitem{KovacsAttila03} A. Kov\'acs. Number expansions in lattices. Math. Comput. Modelling, 38(7-9):909–-915, 2003.
Hungarian applied mathematics and computer applications.

\bibitem{KrStVa}  Z. Krčmáriková, W. Steiner and T. Vávra,   Finite beta-expansions with negative bases, Acta Math. Hungarica,  vol. 152, 485–-504,  2017. 

\bibitem{Penney} W. Penney, A binary system for complex numbers, J. Assoc. Comput. Mach. 12 (1965) 247–248.


 \bibitem{Petho91} A. Peth\H{o}, On a polynomial transformation and its application to the construction of a public keycryptosystem, in: A. Peth\H{o}, M. Pohst, H.G. Zimmer, H.C. Williams (Eds.), Computational NumberTheory, Proc., Walter de Gruyter, Berlin, 1991, pp. 31–-44. 


\bibitem{VavraIsrael} T. V\'avra. Periodic representations in Salem bases, to appear in Israel Journal of Math., {\url https://arxiv.org/abs/1812.08228}.


\bibitem{Vince93a} A. Vince. Radix representation and rep-tiling. In Proceedings of the Twenty-fourth South-
eastern International Conference on Combinatorics, Graph Theory, and Computing (Boca
Raton, FL, 1993), volume 98, pages 199–212, 1993.

\bibitem{Vince93b} A. Vince. Replicating tessellations. SIAM J. Discrete Math., 6(3):501–521, 1993.

\end{thebibliography}
\end{document}